\documentclass{amsart}
\usepackage{amssymb,amsmath,amscd,amsthm}
\usepackage{amsfonts,epsfig,latexsym,graphicx,amssymb}

\usepackage{hyperref}
\usepackage{todonotes}

\newtheorem{Lm}{Lemma}
\newtheorem{Thm}[Lm]{Theorem}
\newtheorem{Prop}[Lm]{Proposition}

\newtheorem{Cor}{Corollary}

\newtheorem{Conj}{Conjecture}

\theoremstyle{definition}
\newtheorem{Def}{Definition}
\newtheorem{Rem}{Remark}
\newtheorem{Ex}{Example}

\def\E{\mathbb{E}}

\def\eps{\varepsilon}
\renewcommand\P{\mathbb{P}}

\def\Z{\mathbb{Z}}
\def\N{\mathbb{N}}
\def\T{\mathbb{T}}

\def\tz{\widetilde{z}}

\def\dist{\mathop{\mathrm{dist}}}

\def\Leb{\mathop{\mathrm{Leb}}}
\def\const{\mathop{\mathrm{const}}}
\def\Sc{S^1}

\newcommand{\Ind}{\mathbf{I}}
\newcommand{\mL}{\mathcal{L}}

\def\bdef{\begin{Def}}
\def\endef{\end{Def}}
\def\bthm{\begin{Thm}}
\def\ethm{\end{Thm}}
\def\bprop{\begin{Prop}}
\def\enprop{\end{Prop}}
\def\blm{\begin{Lm}}
\def\elm{\end{Lm}}
\def\bcor{\begin{Cor}}
\def\ecor{\end{Cor}}
\def\brm{\begin{Rem}}
\def\erm{\end{Rem}}
\def\bfig{\begin{picture}}
\def\efig{\end{picture}}
\def\beq{\begin{eqnarray}}
\def\eneq{\end{eqnarray}}
\def\beal{\begin{aligned}}
\def\enal{\end{aligned}}

\newcommand{\symdiff}{\mathop{\Delta}}
\newcommand{\ddelta}{\delta}
\newcommand{\bm}{\overline{\mu}}
\newcommand{\supp}{\mathop{\mathrm{supp}}}

\title{Synchronization properties of random piecewise isometries}

\author[V. Kleptsyn]{Victor Kleptsyn}

\address{CNRS, Institute of Mathematical Research of Rennes, IRMAR, UMR 6625 du CNRS}

\email{kleptsyn@gmail.com}

\thanks{V.K. was supported in part by RFBR project 13-01-00969-a and RFBR/CNRS joint project 10-01-93115-CNRS\_a}

\author[A. Gorodetski]{Anton Gorodetski}

\address{Department of Mathematics, University of California, Irvine CA 92697, USA}

\email{asgor@math.uci.edu}

\thanks{A.\ G.\ was supported in part by NSF grants DMS-1301515 and IIS-1018433.} 

\date{\today}

\begin{document}

\begin{abstract}
We study the synchronization properties of the random double rotations on tori. We give a criterion that show when synchronization is present in the case of random double rotations on the circle and prove that it is always absent in dimensions two and higher.
\end{abstract}

\maketitle

\section{Introduction}

The observation of a synchronization effect goes back at least to 17th century, when Huygens~\cite{Hu} discovered the synchronization of two linked pendulums. Since then, synchronization phenomena have been observed in numerous systems and settings, see \cite{PR} for a comprehensive survey of the subject. In the theory of dynamical systems, synchronization usually refers to \emph{random} dynamical system trajectories of different initial points converging to each other under the application of a sequence of random transformations. A first such result is the famous Furstenberg's Theorem~\cite{Fur}, stating that under some very mild assumptions, the angle (mod~$\pi$) between the images of any two vectors under a long product of random matrices (exponentially) tends to zero. Projectivizing the dynamics, it is easy to see that this theorem in fact states that random trajectories of the quotient system on the projective space (exponentially) approach each other.

For random dynamical systems on the circle, several results are known. Surely, in random projective dynamics there is a synchronization due to the simplest possible case of Furstenberg's Theorem. In 1984, this result was generalized to the  setting of homeomorphisms (with some very mild and natural assumptions of minimality of the action and the presence of a North-South map) in the work of a physicist V.\,A.\,Antonov \cite{A}, motivated by questions from celestial mechanics. Unfortunately, this work stayed unnoticed by the mathematical community for a long time. Antonov's theorem was later re-discovered in~\cite{KN} (see also the exposition in~\cite{GGKV}). It was further generalized for a non-minimal dynamics in~\cite{DKN}.

The local behavior of (random) orbits of a smooth system is governed by (random) Lyapunov exponents, and their negativity implies at least local synchronization. A theorem of P.\,Baxendale~\cite{Bax} states that for a $C^1$-random dynamical system on a compact manifold without a common invariant measure, there exists an ergodic stationary measure with negative volume Lyapunov exponent. In the one-dimensional case, as there is only one Lyapunov exponent, this implies local contraction, also establishing the exponential speed of contraction in Antonov's Theorem under the additional $C^1$-smoothness assumption. Related results were also obtained in \cite{H1}. These statements have analogues for Riemannian transversely-conformal foliations of compact manifolds, when the random long composition is replaced by the holonomy map along a long random leafwise Brownian path, see~\cite{DK}.

Similar phenomena appear in partially hyperbolic dynamics. Namely, in many cases, points on a generic central leaf tend to each other under the dynamics of the map. This leads to appearance of non-absolutely continuous central foliations. Initially this phenomenon was found by Ruelle, Shub, and Wilkinson \cite{SW, RW} for a perturbation of the product of an Anosov map by an identity map on the circle (see also \cite{H2}), and later was observed in many other partially hyperbolic systems \cite{BB, HP, PT, PTV, SX, V}. The major mechanism that explains ``synchronization'' along the central leaves in partially hyperbolic dynamics is non-vanishing Lyapunov exponents along central leaves (see \cite{GT} for a recent survey of this area).

In some examples the synchronization effect takes place in a slightly different form: a generic pair of points spend most of the time near each other, but sometimes they diverge sufficiently far apart (so that one cannot avoid time-averaging in the description of the synchronization). Such a behavior takes place even in non-random dynamical systems; in particular, it was obtained for some non-strictly expanding interval or circle maps \cite{I}, for the Cherry flow \cite{SV, K2}, and for the modified Bowen example \cite{K1}. Theorem \ref{t:sync} in this paper provides another example of the behavior of this type.

In this paper, we study the synchronization properties of the random piecewise isometries. Dynamical properties of piecewise isometries have been attracting attention lately. They form a surprisingly nontrivial class of dynamical systems even in dimension one, see \cite{BK, B, BT, SIA, Vo, Zh} for the case of one-dimensional piecewise isometries, and \cite{AG, AF, AKMPST, ANST, CGQ, G, GP, LV, TA} for the higher-dimensional case. These maps appear naturally in some applications \cite{ADF, D, SAO, W}, are related to dynamics of polygonal billiards \cite{BK,S}, serve as model problems for some non-linear systems \cite{As}, and also appear as a limit of some renormalization processes introduced to understand complicated non-linear systems \cite{DS}. We will consider the specific case of random double rotations on the circle and higher dimensional analogs. The double rotation of the circle is a map which acts as a rotation on a subset of the circle, and as a different rotation on the complement of the subset. Even in the case when the subset of the circle is an interval the dynamics of a double rotation can be quite non-trivial, see \cite{SIA, BC}. A detailed survey on double rotations can be found in \cite{C}.

Notice that in the case when synchronization is established via negative Lyapunov exponents, smoothness of the systems under consideration is crucial. The proof of Antonov's Theorem uses essentially the fact that homeomorphisms preserve the order of points of the circle. Thus, neither of these techniques is applicable to the case of double rotations.

However, we show that under certain reasonable assumptions the effect of the synchronization is present in the random double rotations on the circle. The mechanism ensuring its presence is thus different (see discussion after Conjecture~\ref{c.1}). Additionally, we show that synchronization seems to be essentially related to one-dimensionality of the phase space. Indeed, Theorem \ref{t:no-sync-2} claims that in higher-dimensional analogous setting synchronization is absent.

\section{Main Results}

\subsection{General setting}

Our results are devoted to the particular case of the dynamics on the $k$-torus, piecewise formed by the translation maps. Namely, take a set $A\subset \T^k$ and a pair of vectors $v_1,v_2\in\T^k$, and consider on the $k$-torus $\T^k$ the map
\begin{equation}\label{eq:fold}
f(x)=\begin{cases} x+v_1, & x\in A,\\
x+v_2, & x\notin A.
\end{cases}
\end{equation}

We consider iterations of $f$, between which the torus is shifted by a random vector, these vectors being chosen w.r.t. the Lebesgue measure and independently on different steps. Formally speaking, we set $\Omega=(\T^k)^{\N}$ to be the set of sequences of elements of $\T^k$, equipped with the probability measure $\P=\Leb^{\N}$, and associate to a sequence $w=(w_i)_{i\in\N}\in \Omega$ the sequence of random iterations
\begin{equation}\label{eq:rds}
F_w^n=f_{w_n}\circ\dots\circ f_{w_1},
\end{equation}
where $f_u(x)=T_u\circ f(x)$, and $T_u(x)=x+u$ is the translation by $u\in\T^k$. We describe the behavior of such dynamics, that (quite unexpectedly!) turns out to be different in the cases $k=1$ and $k>1$.

It is clear that if we compose some fixed translation $T_{u_0}$ with all the random maps $f_{w_i}$, we will get the same random dynamical system (since the Lebesgue measure is a stationary measure which is invariant under the translation $T_{u_0}$). In particular, if we apply $T_{-v_2}$ and set $v=v_1-v_2$, we can turn  the system (\ref{eq:fold}) into the system where
\begin{equation}\label{eq:f}
f(x)=\begin{cases} x+v, & x\in A,\\
x, & x\notin A.
\end{cases}
\end{equation}
{\it In what follows we will always consider the system (\ref{eq:f}) with the following {\bf standing assumption}: components of the vector $v$ together with 1 form a linearly independent $k+1$-tuple over $\mathbb{Q}$ (or, equivalently, the translation $T_v:\mathbb{T}^k\to \mathbb{T}^k$ is minimal).}

An important ingredient of our studies is the following function that we associate to the set~$A$:
\begin{Def}
Let $A\subset \T^k$ be a Borel subset. Its \emph{displacement function} $\varphi_A:\mathbb{T}^k\to \mathbb{R}$ is defined as
$$
\varphi_A(\eps):=\Leb(A\symdiff T_{\eps}(A)).
$$
We say that $A$ \emph{has no translational symmetries}, if
$$
\varphi_{A}(\eps) =0 \quad \Leftrightarrow \quad \eps=0.
$$
\end{Def}
Note that standard measure theory arguments easily imply that $\varphi_{A}$ is a continuous function of~$\eps$ (to show that it does not have discontinuities with oscillation $>\delta$ it suffices to approximate $A$ up to $\frac{\delta}{2}$-measure set by a finite union of rectangles).

\subsection{One-dimensional case, synchronization}

We start with the results on one-dimensional case, i.e. for random double rotations on the circle. It turns out that the integrability of $\frac{1}{\varphi_A(\eps)}$ distinguishes between two possible behaviors, described in Theorem  \ref{t:sync} and Theorem \ref{t:no-sync} below. We start with the non-integrable case; in particular, this is the case if $A$ is a union of $l\ge 1$ intervals: in this case, $\varphi_A(\eps)\sim l\eps$ as $\eps\to 0$ and the system manifests (one of the forms of) synchronization:

\begin{Thm}\label{t:sync}
Let $A\subset \Sc$ be a Borel set that has no translational symmetries, and assume that
\begin{equation}\label{eq:non-int}
\int_{\Sc} \frac{1}{\varphi_A{(\eps)}}\, d\eps =+\infty.
\end{equation}
Then for any $x,y\in \Sc$ for almost any sequence $w\in\Omega$ of iterations one has
\begin{equation}\label{eq:ineq}
\forall \delta>0 \quad \frac{1}{N} \# \{n\in \{1,\dots, N\} \mid \dist(F_w^n(x),F_w^n(y))<\delta \} \to 1 \quad \text{ as } N\to\infty.
\end{equation}
\end{Thm}
\begin{Rem}
Notice that (\ref{eq:ineq}) is equivalent to
$$
\frac{1}{N}\sum_{n=1}^N\dist(F_w^n(x),F_w^n(y))\to 0 \ \ \text{as}\ \ N\to \infty.
$$
\end{Rem}

This statement also has interesting consequences for iterations of the Lebesgue measure. To state it, we will need a way (which is one of many equivalent ways) to measure non-Diracness of a measure on the circle.
\begin{Def}
For any measure $m$ on the circle let
$$D(m):=\iint \dist(x,y) \, dm(x) \, dm(y).$$
\end{Def}

Then, we have the following

\begin{Thm}\label{t:Leb-sync}
Under the assumptions of Theorem~\ref{t:sync},
\begin{equation}\label{eq:Leb-fwd}
\forall \delta>0 \quad \P(D((F_w^n)_* \Leb)>\delta) \to 0 \quad \text{as } n\to\infty.
\end{equation}
\end{Thm}
Note that this conclusion is  stronger than the one that can be obtained as an immediate corollary of Theorem~\ref{t:sync} by averaging on $x$ and $y$. Indeed, in the latter case we would have a convergence to zero only of Chesaro averages of the probabilities in~\eqref{eq:Leb-fwd}, and the theorem states convergence for the probabilities themselves.

In the theory of random dynamical systems, often together with the usual order of composition are considered the reversed-order compositions, with the next map being applied first:
\begin{Def}
For $w\in\Omega$, let
$$
F_{w,rev}^n:= f_{w_1}\circ\dots\circ f_{w_n}.
$$
\end{Def}
This order of composition has an advantage that in this case it is more likely that an individual sequence of images of a given measure converges. 
 And indeed, this is what happens in our case.

\begin{Thm}\label{t:rev}
Under the assumptions of Theorem~\ref{t:sync}, there exists a measurable map $L:\Omega\to \Sc$, such that almost surely
\begin{equation}\label{eq:Leb-conv}
(F_{w,rev}^n)_* \Leb \to \ddelta_{L(w)} \quad \text{as } n\to\infty.
\end{equation}
\end{Thm}

Let us now consider the particular case of $A$ being an interval. Note that for the reversed-order composition we can define the (random) ``topological attractor''. Namely, for any $w\in \Omega$ the sequence $\overline{F_{w,rev}^n(\Sc)}$ is a nested sequence of nonempty closed sets, and thus has a non-empty intersection:
\begin{Def}
The \emph{topological attractor} is a random set $X=X(w)$, defined as
$$
X(w)=\bigcap_{n\in\N} \overline{F_{w,rev}^n(\Sc)}
$$
\end{Def}
However, even though on every finite step, assuming that $A$ is an interval or a finite union of intervals, one has $\overline{F_{w,rev}^n(\Sc)}=\supp (F_{w,rev}^n)_*\Leb$, passing to the limit is not immediate. Namely, a conclusion
$$
\supp \lim_{n\to\infty} (F_{w,rev}^n)_*\Leb \subset X(w)
$$
is immediate, but equality is not at all guaranteed. The absence of such equality would mean that, even though \emph{most} of the initial Lebesgue measure after a large number of iterations is most probably concentrated near one point, there will be (once again, most probably) some parts of its image that are far away from this point. And quite interestingly, it seems that indeed such an effect takes place: in this random system there is a difference between the ``topological'' and ``measurable'' limit behaviors. Namely,  numerical simulations, as well as some \emph{very} rough heuristic arguments, predict the following
\begin{Conj}\label{c.1}
If $A$ is an interval, the topological attractor $X(w)$ is almost surely a Cantor set.
\end{Conj}

We would like to conclude the statement of results for this case by pointing out an interesting, and quite instructive, parallel of the synchronization observed here to some effects already known in the ordinary (non-random) dynamical systems. Namely: note that even though Theorem~\ref{t:sync} states that any two points spend \emph{most} of the time close to each other, the distance between them never (except for a case of two points on the same orbit of $T_{v}$) converges to zero. Indeed, the difference $x-y$ between the two points can be changed after one iteration only either by adding or by subtracting the fixed vector~$v$.

This is exactly the type of a situation that happens, for instance, in the Cherry flow (see~\cite{SV},~\cite{K2}), as well as for the separatrix loop or modified Bowen's example (see~\cite{K1}): the proportion of time spent by any individual point near the saddle becomes closer and closer to~$1$, even if from time to time the point of the orbit leaves the neighborhood of the saddle (only to get ``stuck'' there for a longer time after it comes back even closer). Perhaps even more instructive analogy is the example of a non-strictly expanding circle diffeomorphism with one neutral fixed point. Then, everything depends on the speed of repulsion at this point: as it follows from Inoue's results~\cite{I}, if the expanding map at this point behaves as $x\mapsto x (1+|x|^d)$ with $d>1$, the exit-time is Lebesgue-non-integrable, and only SRB measure is concentrated at the fixed point (see also~\cite{DKN2}). On the other hand, if $d<1$, the exit time is Lebesgue-integrable, and the iterations of the Lebesgue measure tend to an absolutely continuous invariant measure; this is in exact parallel with Theorems~\ref{t:no-sync} and~\ref{t:no-sync-2} below.

\subsection{One-dimensional case, no synchronization}

The non-integrability~\eqref{eq:non-int} turns out not only to be a sufficient, but also a necessary condition for the synchronization:

\begin{Thm}\label{t:no-sync}
Let $A\subset \Sc$ be a Borel set that has no translational symmetries, and assume that
$$
\int_{\Sc} \frac{1}{\varphi_A{(\eps)}}\, d\eps <+\infty.
$$
Then for Lebesgue-almost any $x,y\in S^1$ and for almost any sequence $w\in\Omega$ of iterations one has a weak convergence of measures
$$
\frac{1}{N} \sum_{n=1}^N \ddelta_{F_w^n(x)-F_w^n(y)} \to \bm,
$$
where the probability measure $\bm=\frac{1}{Z}\cdot \frac{dx}{\varphi_A(x)}$, \, $Z=\int_{\Sc} \frac{dx}{\varphi_A(x)}$ does not depend on $w$.
\end{Thm}

In particular, as the measure $\bm$ does not charge $0$, it is immediate to say that there is no synchronization in any system satisfying
the assumptions of Theorem~\ref{t:no-sync}. The following example shows that such a behavior is possible even for not-so-bad set~$A$.

\begin{Ex}
Let $A\subset [0,1]$ be the Cantor set of positive measure, constructed $[0,1]$ in the following way. Take $M_0=[0,1]$ and construct for each $n\in \mathbb{N}$ the set $M_n$  by removing $2^{n-1}$ intervals of length $4^{-n}$ from $M_{n-1}$, centered at the middle of intervals of~$M_{n-1}$. Set $A=\cap_n M_n$. Then, for any sufficiently small $\eps$ we have $\varphi_A(\eps)\gtrsim {\eps}^{2/3}$; in particular, the function $\frac{1}{\varphi_A(\eps)}$ is integrable. Indeed, let $8^{-n}<\eps<8^{-(n-1)}$, where $n\ge 2$. Then
$$
\Leb(M_n\setminus A)=\sum_{j> n} 2^{j-1} 8^{-j} = \frac{1}{6} 4^{-n}.
$$
On the other hand,
$$
\frac{1}{2}\varphi_A(\eps)=\Leb(A\setminus T_{\eps}(A)) \ge \Leb(M_n\setminus T_{\eps}(M_n)) - \Leb(M_n\setminus A).
$$
Now, it is easy to see that $\Leb(M_n\setminus T_{\eps}(M_n)) \ge 2^{n-1} \cdot 8^{-n} = \frac{1}{2} 4^{-n}$, as at least all the intervals, removed on the $n$th step of construction, contribute to this difference. Thus,
$$
\frac{1}{2}\varphi_A(\eps)\ge \frac{1}{2} 4^{-n}-\frac{1}{6}4^{-n} =\frac{1}{3} 4^{-n} \ge \const \eps^{2/3}.
$$
\end{Ex}

\begin{Rem}
Slightly modifying the construction, one can find Cantor set $A$ with $\varphi_A(\eps)>\const \eps^{\alpha}$ for an arbitrary small $\alpha>0$.
\end{Rem}

\subsection{Higher-dimension case: no synchronization ever}

It turns out that in the higher-dimensional case the synchronization never takes place:
\begin{Thm}\label{t:no-sync-2}
Let $A\subset \T^k$ be a Borel subset that has no translational symmetries, and assume that $k>1$.
Then for Lebesgue-almost any $x,y\in \T^k$ and for almost any sequence $w\in\Omega$ one has a weak convergence of measures
$$
\frac{1}{N} \sum_{n=1}^N \ddelta_{F_w^n(x)-F_w^n(y)} \to \bm,
$$
where the probability measure $\bm$ does not depend on $w$.  Moreover, $\bm$ is given explicitly by
\begin{equation}\label{eq:onemore}
 \bm=\frac{1}{Z}\cdot \frac{dx}{\varphi_A(x)}, \, Z=\int_{T^k} \frac{dx}{\varphi_A(x)}.
\end{equation}
\end{Thm}
The expression (\ref{eq:onemore}) is well defined (as we will see later in the proof of Theorem~\ref{t:no-sync-2}) since in the higher-dimensional situation the integral of $\frac{1}{\varphi_{A}(\eps)}$ \emph{always} converges.

\section{Proofs}
Let us start with the proof of Theorem~\ref{t:sync}:
\begin{proof}[Proof of Theorem~\ref{t:sync}]
Fix the points $x,y\in\Sc$. For a given sequence $w=(w_n)\in \Omega$ and the associated random images $x_n=F_w^n(x)$, $y_n=F_w^n(y)$, consider corresponding  sequence
$$
z_n=F_w^{n+1}(x)-F_w^{n+1}(y)=f(F_w^n(x)) - f(F_w^n(y)).
$$
of differences between these random images (the last equality is due to the fact that the rotation by $w_{n+1}$ does not change the difference). A key remark is that the process $(z_n)_{n=1}^{\infty}$ is a stationary Markov process, and is governed by the following transitional probabilities:
\begin{equation}\label{eq:proc}
z_{n+1}=\begin{cases}
z_n+v, & \text{ with probability } \frac{1}{2}\varphi_A(z_n),\\
z_n-v, & \text{ with probability } \frac{1}{2}\varphi_A(z_n),\\
z_n, & \text{ with probability } 1-\varphi_A(z_n).
\end{cases}
\end{equation}

Indeed, for known $w_1,\dots,w_n$ (and thus $z_n$), the conditional distribution of $x_{n+1}=F_w^{n+1}(x)$ is independent of them and is given by the Lebesgue measure: $x_{n+1}=T_{w_{n+1}}(f(F_w^{n}(x)))$, and $w_{n+1}$ is distributed w.r.t. the Lebesgue measure.
Now, notice that the application of $f$ changes the difference vector between the two points in the following way:
$$
f(a)-f(b) = \begin{cases}
 (a-b)+v, & a\in A, \, b\notin A,\\
(a-b)-v, & a\notin A, \, b\in A,\\
(a-b), & \text{otherwise}.
\end{cases}
$$
In particular,
\begin{equation}\label{eq:cz}
z_{n+1}=
f(x_{n+1})-f(y_{n+1}) = \begin{cases}
 z_n+v, & \text{ if } x_{n+1}\in A \setminus (T_{-z_n}(A)) \\
z_n-v, & \text{ if }x_{n+1}\in (T_{-z_n}(A)) \setminus A \\
z_n & \text{otherwise}.
\end{cases}
\end{equation}
The conditions on $x_{n+1}$ here come from the fact that $y_{n+1}=x_{n+1}+z_n$, hence $y_{n+1}\in A$ if and only if $x_{n+1}\in T_{-z_n}(A)$. Finally, we notice that $\Leb(A)=\Leb(T_{-z_n}(A))$ and hence
$$
\Leb(A \setminus (T_{-z_n}(A))) = \Leb ((T_{-z_n}(A)) \setminus A) = \frac{\Leb (A \symdiff T_{z_n}(A))}{2} =\frac{1}{2} \varphi_A(z_n).
$$
This concludes the proof of~\eqref{eq:proc}.

Now, let us change the point of view on the Markov process~\eqref{eq:proc}. Given $z_0=x-y$, we first consider a simple random walk $(c_j)$ on $\Z$, taking
$$
c_0=0, \quad c_{j+1}=\begin{cases}
 c_j+1& \text{with probability } 1/2, \\
 c_j-1& \text{with probability } 1/2.
\end{cases}
$$
Consider an auxiliary process $\tz_j=z_0+c_j v$. We claim that the Markov process~\eqref{eq:proc} can be seen as a ``slowing down'' of the process $\tz$. Namely, take any its trajectory~$(\tz_j)$, and consider independent random variables $t_j$ that are distributed geometrically with the mean $\frac{1}{\varphi_A(\tz_j)}$, that is,
\begin{equation}\label{eq:geom}
\P(t_j=i) = (1-q)\cdot q^{i-1}, \quad q=1-\varphi_A(\tz_j), \quad i=1,2,\dots.
\end{equation}
Then, the ``slowed down'' process
\begin{equation}\label{eq:slowing}
Z_n=\tz_{J(n)}, \quad J(n):=\max\{j \mid t_1+t_2+\dots t_j\le n\},
\end{equation}
is a Markov process that has the same law as~$z_n$.

Notice now that the trajectory of the process $\tz_j$ is almost surely asymptotically distributed w.r.t. the Lebesgue measure. Indeed, the trajectory of $\tz_j$ is a trajectory of a random dynamical system, generated by the translation $T_v$ and its inverse, each applied with the probability~$1/2$, and starting at the point~$z_0$. The Lebesgue measure is an ergodic stationary measure of this system. Thus, Kakutani's random ergodic theorem (see~\cite[Theorem 3.1]{Furman}, \cite{Kakutani}) implies that for almost any initial point $z$ almost surely its random trajectory is asymptotically distributed w.r.t. the Lebesgue measure. Now, if for a random sequence of iterations the trajectory of one point $z$ is asymptotically distributed w.r.t. the Lebesgue measure, then the same holds for any other point $z'$ (in particular, for $z'=z_0$): these two trajectories differ by a translation~$T_{z'-z}$.

At the same time, due to our assumptions the only zero of the function $\varphi_A(z)$ is $z=0$, and for the integral of the expectation we have
$$
\int_{\Sc} \frac{1}{\varphi_A(z)} dz= +\infty.
$$

This motivates the following
\begin{Lm}\label{l:sums}
Under the assumptions of Theorem~\ref{t:sync}, conditionally to any trajectory $\{\tz_j\}$ that is asymptotically distributed w.r.t. the Lebesgue measure (that is, $\frac{1}{j}\sum_{i=1}^j \delta_{\tz_j}$ weakly converges to~$\Leb$), almost surely
 the number of iterations of the slowed down process grows superlinearily,
\begin{equation}\label{eq:t-avg}
\frac{t_1+t_2+\dots+t_j}{j}\to +\infty\,  \text{ as } j\to \infty.
\end{equation}
while at the same time the number of iterations spent outside of any $\delta$-neighborhood of $z=0$ grows linearly:
\begin{equation}\label{eq:t-out-avg}
\forall \delta>0 \quad \frac{1}{j} \sum_{i=1}^j t_j\cdot \Ind_{S^1\backslash U_{\delta}(0)}(z_j) \to \int_{\Sc\setminus U_{\delta}(0)} \frac{1}{\varphi_A(z)}\, dz  <+\infty \quad \text{ a.s. as } j\to\infty,
\end{equation}
where $\Ind_B$ detotes the indicator function of the set~$B$.
\end{Lm}

Notice that Lemma \ref{l:sums} will immediately imply the statement of Theorem \ref{t:sync}. Indeed, as we mentioned already, for any initial points $x,y$ (and hence for any $\tz_0=z_0=x-y$) the trajectory $\tz_j$ is almost surely distributed w.r.t. the Lebesgue measure (here we assume that $x,y$ are not on the same trajectory of $T_v$, otherwise the claim of Theorem \ref{t:sync} trivially holds). Lemma~\ref{l:sums} then states that the convergences~\eqref{eq:t-avg} and~\eqref{eq:t-out-avg} hold, and dividing the latter by the former, we see that almost surely
\begin{equation}\label{eq:conv-to-0}
\frac{\sum_{i=1}^j t_j\cdot \Ind_{z_j\notin U_{\eps}(0)}  }{\sum_{i=1}^j t_j }  = \frac{\frac{1}{j}\sum_{i=1}^j t_j\cdot \Ind_{z_j\notin U_{\eps}(0)}  }{\frac{1}{j} \sum_{i=1}^j t_j }  \to 0 \quad \text{as } j\to\infty.
\end{equation}
In other words, the proportion of time spent by the slowed down process $Z_n$ outside of any $U_{\delta}(0)$ tends to zero as~$n\to\infty$. Formally speaking,~\eqref{eq:conv-to-0} is such a convergence for a subsequence of times $n$ that are of the form $n=t_1+\dots+t_j$, but as $Z$ does not change between such moments of time, it suffices to state the full convergence. (This is the same type of the argument that was used in~\cite{I}, \cite{K1}, \cite{SV})


In fact, proving both convergences for Lebesgue-almost any initial $z_0$ (and hence for almost any pair $x,y$ of initial points) would be a bit easier, as both~\eqref{eq:t-avg} and~\eqref{eq:t-out-avg} would follow from the Birkhoff ergodic theorem (see an analogous estimate in the proof of Theorem~\ref{t:no-sync} below). However, as we want to prove the almost sure synchronization for any pair of initial points $x,y$, we will have to make some more technical estimates.

\begin{proof}[Proof of Lemma~\ref{l:sums}]
Let us start with~\eqref{eq:t-out-avg}. Note that for any $\delta>0$ all random variables $t_j\cdot \Ind_{\tz_j\notin U_{\delta}(0)}$ have uniformly bounded dispersion: indeed, they are either identically zero, or geometric with a uniformly bounded expectation (the function $\varphi_A(z)$ is continuous and $z=0$ is its only zero). Hence, the difference
$$
\frac{1}{j} \left(\sum_{i=1}^j t_i \Ind_{\tz_j\notin U_{\delta}(0)} - \sum_{i=1}^j \E (t_i \Ind_{\tz_j\notin U_{\delta}(0)}) \right)
$$
almost surely tends to zero (see e.g.~\cite[Theorem 2.3.10]{SS}). Now, we have
$$
\frac{1}{j} \sum_{i=1}^j \E (t_i  \Ind_{\tz_j\notin U_{\delta}(0)}) =\frac{1}{j} \sum_{i=1}^j \frac{\Ind_{\tz_j\notin U_{\delta}(0)}}{\varphi_A(\tz_j)} \to \int_{\Sc\setminus U_{\delta}(0)} \frac{1}{\varphi_A(z)} dz \quad \text{as } j\to\infty,
$$
where the convergence in the right hand side follows from the asymptotic distribution of the trajectory~$\tz_j$. This concludes the proof of~\eqref{eq:t-out-avg}.

Now, for any $\delta$ we have
\begin{equation}\label{eq:l-int}
\liminf_{j\to\infty} \frac{t_1+t_2+\dots+t_j}{j} \ge \liminf_{j\to\infty} \frac{1}{j} \sum_{i=1}^j t_j\cdot \Ind_{z_j\notin U_{\delta}(0)} = \int_{\Sc\setminus U_{\delta}(0)} \frac{1}{\varphi_A(z)}\, dz.
\end{equation}
As $\delta>0$ can be chosen arbitrarily small, the integral in the right hand side of~\eqref{eq:l-int} can be made arbitrarily large (as the integral $\int_{\Sc} \frac{1}{\varphi_A(z)} dz $ diverges due to the assumptions of the theorem). Thus, the limit in the left hand side of~\eqref{eq:t-avg} is infinite.

This concludes the proof of Lemma \ref{l:sums}, and hence of Theorem~\ref{t:sync}.

\end{proof}
\end{proof}

\begin{proof}[Proof of Theorem~\ref{t:no-sync}]
First, notice that the random walk $\tz_j$ can be modeled by a map
\begin{equation}\label{eq:rw}
G_0:\Sc\times \{+1,-1\}^{\N} \to \Sc\times \{+1,-1\}^{\N}, \quad G_0(\tz,(c_j)) = (\tz+ v c_1, (c_{j+1}))
\end{equation}
with an ergodic measure $\nu$ which is a product of Lebesgue measure on $\Sc$ and the Bernoulli measure on $\{+1,-1\}^{\N}$ (i.e. on the space of sequences  $\{(c_j)_{j\in \mathbb{N}}\}$).

Though, to slow down the random walk $\tz_j$, we need the associated geometric distributions. To model these inside a skew product, we will use a standard argument: any distribution can be realized on $([0,1],\Leb)$. Define $\psi:[0,1]\times [0,1] \to \N$ as
$$
\psi(q,s)=k \quad \text{if } s\in[1-q^{k-1},1-q^k);
$$
then, for any $q\in (0,1)$ the random variable $\psi(q,\cdot):([0,1],\Leb)\to \N$ has a geometric distribution with the expectation~$\frac{1}{1-q}$.

Take the product $G$ of the system~\eqref{eq:rw} with the Bernoulli shift on the space~$[0,1]^{\N}$. It has a natural invariant measure $\overline{\nu}=\nu\times \Leb^{\N}$. Consider the function
$$
T: (\Sc\times \{+1,-1\}^{\N} )\times [0,1]^{\N} \to \N, \quad T(\tz,(c_j), (s_j)) = \psi\left(1-\varphi_A(\tz),s_1\right).
$$
Then,
\begin{equation}\label{eq:comp}
T\circ G^j = \psi \left(1-\varphi_A(\tz_{j}),s_{j+1}\right),
\end{equation}
where $\tz_{j}=\tz+ v(c_1+\dots+c_j)$. Note, that conditionally to any base point $(\tz,(c_j))$ --or, what is the same, to the associated random walk trajectory $(\tz_j)$,-- the functions $T$, $T\circ G$, \dots, $T\circ G^n$,\dots are independent as random variables (they depend on different $[0,1]$-coordinates $s_1,s_2,\dots, $ respectively), and have exactly the required distribution of times $t_1$, $t_2$,\dots.

The system $G$ is a product of an ergodic $G_0$ with a mixing Bernoulli shift, and hence is also ergodic. Thus, the application of the Birkhoff ergodic theorem to the function $T$ immediately implies
\begin{equation}\label{eq:t-int}
\frac{1}{j} (t_1+t_2+\dots+t_j) \to \int T d\,\overline{\nu}= \int_{\Sc} \frac{1}{\varphi_A(z)} dz=: Z.
\end{equation}
almost surely for almost every $\tz_0$.

Now, for any interval $J\subset\Sc$ in the same way one has almost surely
\begin{equation}\label{eq:t-int-J}
\frac{1}{j} \left(\sum_{i=1}^j t_i \Ind_{\tz_i\in J}\right) \to \int T \Ind_{\tz_1\in J} d\,\overline{\nu} = \int_{J} \frac{1}{\varphi_A(z)} dz.
\end{equation}

Thus, for the slowed down process for the subsequence of moments $N_j:=t_1+\dots +t_j$  
we have
\begin{multline}\label{eq:lim-Nj}
\frac{1}{N_j} \#\{n\le N_j \mid z_n\in J\} = \frac{1}{N_j} \sum_{i=1}^j  (t_i \Ind_{\tz_i\in J}) = \frac{1}{N_j/j}  \cdot \frac{1}{j} \sum_{i=1}^j  (t_i \Ind_{\tz_i\in J}) \to \\
\to \frac{1}{Z} \cdot \int_{J} \frac{1}{\varphi_A(z)} dz,
\end{multline}
where the limit for the first factor comes from~\eqref{eq:t-int}, and for the second factor from~\eqref{eq:t-int-J}. Now,~\eqref{eq:t-int} implies that almost surely $t_j=o(j)$, thus allowing to extend~\eqref{eq:lim-Nj} from the subsequence $N_j$ to all the natural numbers:
$$
\frac{1}{N} \#\{n\le N \mid z_n\in J\} \to \frac{1}{Z} \cdot \int_{J} \frac{1}{\varphi_A(z)} dz \quad \text{as } N\to\infty.
$$
This implies that $z_n$ are asymptotically distributed w.r.t. the measure $\frac{1}{Z}\cdot \frac{dz}{\varphi_A(z)}$, thus concluding the proof of Theorem \ref{t:no-sync}.

\end{proof}

\begin{proof}[Proof of Theorem~\ref{t:no-sync-2}]
The argument in higher dimension repeats verbatim the proof of Theorem~\ref{t:no-sync}. The only difficulty is that we do not anymore assume that the integral converges, thus we have to prove it. This is done by the following
\begin{Lm}\label{l:lower-phi}
Let $A\subset \T^k$ be a Borel subset that admits no translational symmetries. Then, there exists $\alpha>0$ such that
$$
\forall u\in \T^k  \quad \varphi_A(u)\ge\alpha \dist(u,0).
$$
\end{Lm}
\begin{proof}[Proof of Lemma \ref{l:lower-phi}]
Notice first that the function $\varphi_A$ is subadditive:
\begin{multline*}
\varphi_A(u+u') = \Leb(A\symdiff T_{u+u'}(A)) \le  \Leb(A\symdiff T_{u}(A)) + \Leb(T_u(A)\symdiff T_{u+u'}(A)) = \\ = \Leb(A\symdiff T_{u}(A)) + \Leb(T_u(A\symdiff T_{u'}(A))) = \varphi_A(u) + \varphi_A(u').
\end{multline*}

Now, assume that Lemma \ref{l:lower-phi} does not hold. Then, there exists a family of vectors $u_n$ such that $\varphi_A(u_n)<\frac{1}{n} \dist(u_n,0)$. The absence of the translational symmetries implies that $\varphi_A$ is bounded away from zero outside any neighborhood of~$0$, and hence one should have $u_n\to 0$. Now, consider the sequence of vectors $V_n=N_n\cdot u_n$, where $N_n=[\frac{1}{2\dist(u_n,0)}]$. We have
\begin{equation}\label{eq:Vn}
\varphi_A(V_n)\le N_n \varphi_A(u_n)\le \frac{1}{2\dist(u_n,0)} \cdot \frac{1}{n}\dist(u_n,0) \to 0 \quad \text{as } n\to\infty.
\end{equation}
On the other hand, $\dist(V_n,0)\to \frac{1}{2}$. Extracting any convergent subsequence $V_{n_i}\to V$, we find a vector $V$, for which $\dist(V,0)=\frac{1}{2}$ and $\varphi_A(V)=0$ due to~\eqref{eq:Vn}. This contradicts the assumption of the absence of the translational symmetries, and this contradiction concludes the proof.
\end{proof}

The lower bound from Lemma~\ref{l:lower-phi} implies that the singularity at $0$ of the integral $\int_{\T^k} \frac{1}{\varphi_A(z)}dz$ is at most of order $\frac{1}{\dist(z,0)}$, and hence the integral converges. The arguments of the proof of Theorem~\ref{t:no-sync} are then applicable verbatim.
\end{proof}

\begin{Rem}
Notice that the same arguments as in the proof of Lemma \ref{l:lower-phi} show that for any set $A$ either the conclusion of Lemma \ref{l:lower-phi} holds in a neighborhood of $0$, or $A$ admits a one-parametric group of translational symmetries.
\end{Rem}

\begin{proof}[Proof of Theorem~\ref{t:rev}]
First, note that the Lebesgue measure is the unique \emph{stationary} (i.e., equal to the average of its random images) measure of the considered system. Indeed, for \emph{any} measure $m$, one has
$$
\int_{\Sc} (T_u f)_* m \, d\Leb(u) = (f_*m)* \Leb = \Leb,
$$
where $\mu_1*\mu_2$ stays for the convolution of measures $\mu_1$ and $\mu_2$.

Next, a famous argument in the study of random dynamical systems, going back to Furstenberg (see~\cite[Corollary to Lemma~3.1 \& Proposition~3.4,  p.~20]{Fur2}), is that the sequence of iterations with reversed order of a stationary measure forms a martingale (with values in the space of measures), and hence converges almost surely. Hence, there exists a measurable map $\mL$ from $\Omega$ to the space of probability measures on $\Sc$, such that almost surely
$$
(F_{w,rev}^n)_* \Leb\to \mL(w) \quad \text {as } n\to\infty.
$$

On the other hand, Theorem~\ref{t:sync} implies that for any two $x,y\in\Sc$ almost surely
$$
\forall \delta>0 \quad \frac{1}{N} \# \{n\in \{1,\dots, N\} \mid \dist(F_w^n(x),F_w^n(y))<\delta \} \to 1 \quad \text{ as } N\to\infty,
$$
which can be rewritten as
$$
\frac{1}{N} \sum_{n=1}^N \dist(F_w^n(x),F_w^n(y)) \to 0 \quad \text{ as } N\to\infty.
$$
Integrating over $x$ and over $y$, we get
$$
\frac{1}{N} \sum_{n=1}^N D((F_w^n)_* \Leb) = \frac{1}{N} \sum_{n=1}^N \iint \dist(F_w^n(x),F_w^n(y)) \, dx\, dy \to 0 \quad \text{ as } N\to\infty.
$$

Now, take the expectation: we get
$$
\frac{1}{N} \sum_{n=1}^N \int D((F_w^n)_* \Leb)  \, d\P(w) \to 0 \quad \text{ as } N\to\infty.
$$
For any fixed $n$, the laws of random compositions $F_w^n$ and $F_{w,rev}^n$ coincide: both are the compositions of $n$ independently chosen maps (the difference comes when we consider the sequence of iterations). Hence,
$$
\int D((F_w^n)_* \Leb)  \, d\P(w)= \int D((F_{w,rev}^n)_* \Leb)  \, d\P(w),
$$
and thus
$$
\frac{1}{N} \sum_{n=1}^N \int D((F_{w,rev}^n)_* \Leb)  \, d\P(w) \to 0 \quad \text{ as } N\to\infty.
$$

But we know that almost surely $(F_{w,rev}^n)_* \Leb\to \mL(w)$, what implies $D((F_{w,rev}^n)_* \Leb) \to D(\mL(w))$ and hence
$$
\frac{1}{N} \sum_{n=1}^N \int D((F_{w,rev}^n)_* \Leb)  \, d\P(w) \to \int D(\mL(w)) \, d\P(w)   \quad \text{ as } N\to\infty.
$$
Thus, the integral $\int D(\mL(w)) \, d\P(w)$ vanishes, and hence $\mL(w)$ is almost surely a Dirac measure:
$$
\mL(w)=\delta_{L(w)}.
$$
This concludes the proof of Theorem \ref{t:rev}.
\end{proof}

\begin{proof}[Proof of Theorem~\ref{t:Leb-sync}]
From Theorem~\ref{t:rev} we have almost surely
$$
(F_{w,rev}^n)_* \Leb\to \delta_{L(w)}  \quad \text{ as } n\to\infty,
$$
and hence almost surely
$$
D((F_{w,rev}^n)_* \Leb)\to 0  \quad \text{ as } n\to\infty.
$$
In particular, for any $\eps>0$
$$
\P(D((F_{w,rev}^n)_* \Leb)>\eps) \to 0  \quad \text{ as } n\to\infty.
$$
However (as we have already seen in the proof of Theorem~\ref{t:rev}), for any $n$ the laws of random compositions $F_w^n$ and $F_{w,rev}^n$ coincide. Hence,
$$
\P(D((F_{w,rev}^n)_* \Leb)>\eps) = \P(D((F_{w}^n)_* \Leb)>\eps),
$$
and thus
$$
\P(D((F_{w}^n)_* \Leb)>\eps) \to 0  \quad \text{ as } n\to\infty.
$$
\end{proof}

\end{document}